\documentclass{amsart}
\usepackage[all]{xy}
\usepackage{verbatim}
\usepackage{color}
\usepackage{amsthm}
\usepackage{amssymb}
\usepackage[colorlinks=true]{hyperref}
\usepackage{footmisc}
\usepackage{stmaryrd}
\numberwithin{equation}{section}

\newtheorem{theorem}[equation]{Theorem}
\newtheorem*{theorem*}{Theorem} \newtheorem{lemma}[equation]{Lemma}

\newtheorem*{conjecture*}{Mamma Conjecture}
\newtheorem*{conjecture1*}{Mamma Conjecture (revisited)}
\newtheorem{proposition}[equation]{Proposition}
\newtheorem{corollary}[equation]{Corollary}
\newtheorem*{corollary*}{Corollary}

\theoremstyle{remark}

\theoremstyle{remark}
\newtheorem{remark}[equation]{Remark}

\newcommand{\cC}{{\mathcal C}}

\newcommand{\cF}{{\mathcal F}}

\newcommand{\cL}{{\mathcal L}}
\newcommand{\cM}{{\mathcal M}}

\newcommand{\cO}{{\mathcal O}}

\newcommand{\cV}{{\mathcal V}}
\newcommand{\cW}{{\mathcal W}}
\newcommand{\cX}{{\mathcal X}}

\newcommand{\cZ}{{\mathcal Z}}

\newcommand{\bbA}{\mathbb{A}}

\newcommand{\bbG}{\mathbb{G}}

\newcommand{\bbP}{\mathbb{P}}

\newcommand{\bbQ}{\mathbb{Q}}
\newcommand{\bbZ}{\mathbb{Z}}

\DeclareMathOperator{\NMot}{NMot}
\newcommand{\dgcat}{\mathrm{dgcat}} % codimension 
\newcommand{\perf}{\mathrm{perf}}

\newcommand{\dg}{\mathrm{dg}}

\newcommand{\uHom}{\underline{\mathrm{Hom}}}
\newcommand{\Hom}{\mathrm{Hom}}

\newcommand{\too}{\longrightarrow}

\newcommand{\ie}{\textsl{i.e.}\ }

\let\oldmarginpar\marginpar
\def\marginpar#1{\oldmarginpar{\tiny #1}}

\begin{document}

\title[Schur-finiteness (and Bass-finiteness) conjecture]{Schur-finiteness (and Bass-finiteness) conjecture \\ for quadric fibrations and \\for families of sextic du Val del Pezzo surfaces}
\author{Gon{\c c}alo~Tabuada}
\address{Gon{\c c}alo Tabuada, Department of Mathematics, MIT, Cambridge, MA 02139, USA}
\email{tabuada@math.mit.edu}
\urladdr{http://math.mit.edu/~tabuada}
\thanks{The author was supported by a NSF CAREER Award}

\date{\today}
\abstract{Let $Q \to B$ be a quadric fibration and $T \to B$ a family of sextic du Val del Pezzo surfaces. Making use of the recent theory of noncommutative mixed motives, we establish a precise relation between the Schur-finiteness conjecture for $Q$, resp. for $T$, and the Schur-finiteness conjecture for $B$. As an application, we prove the Schur-finiteness conjecture for $Q$, resp. for $T$, when $B$ is low-dimensional. Along the way, we obtain a proof of the Schur-finiteness conjecture for smooth complete intersections of two or three quadric hypersurfaces. Finally, we prove similar results for the Bass-finiteness conjecture.}}

\maketitle
%\vskip-\baselineskip
%\vskip-\baselineskip
%\vskip-\baselineskip
%\tableofcontents

%\bigskip

%\medskip
%-------------------------------------------------------------------------------
\section{Introduction}
%-------------------------------------------------------------------------------
\subsection*{Schur-finiteness conjecture}
Let $\cC$ be a $\bbQ$-linear, idempotent complete, symmetric monoidal category. Given a partition $\lambda$ of an integer $n\geq 1$, consider the corresponding $\bbQ$-linear representation $V_\lambda$ of the symmetric group $\mathfrak{S}_n$ and the associated idempotent $e_\lambda \in \bbQ[\mathfrak{S}_n]$. Under these notations, the Schur-functor $S_\lambda\colon \cC \to \cC$ sends an object $a\in \cC$ to the direct summand of $a^{\otimes n}$ determined by $e_\lambda$. Following Deligne \cite[\S1]{Deligne}, $a \in \cC$ is called {\em Schur-finite} if it is~annihilated by some Schur-functor.

Voevodsky introduced in \cite{Voevodsky} a triangulated category of geometric mixed motives $\mathrm{DM}_{\mathrm{gm}}(k)_\bbQ$ (over a perfect base field $k$). By construction, this category is $\bbQ$-linear, idempotent complete, symmetric monoidal, and comes equipped with a $\otimes$-functor $M(-)_\bbQ \colon \mathrm{Sm}(k) \to \mathrm{DM}_{\mathrm{gm}}(k)_\bbQ$ defined on smooth $k$-schemes of finite type. Given $X \in \mathrm{Sm}(k)$, an important conjecture in the theory of motives is the following:

\vspace{0.1cm}

{\bf Conjecture} $\mathrm{S}(X)$: The geometric mixed motive $M(X)_\bbQ$ is Schur-finite.

\vspace{0.1cm}

Thanks to the (independent) work of Guletskii \cite{Guletskii} and Mazza \cite{Mazza}, the conjecture $\mathrm{S}(X)$ holds in the case where $\mathrm{dim}(X)\leq 1$. Thanks to the work of Kimura \cite{Kimura} and Shermenev \cite{Shermenev}, the conjecture $\mathrm{S}(X)$ also holds in the case where $X$ is an abelian variety. Besides these cases (and some other cases scattered in the literature), the Schur-finiteness conjecture remains wide open. 

The main goal of this note is to prove the Schur-finiteness conjecture in the new cases of quadric fibrations and families of sextic du Val del Pezzo surfaces.

\subsection*{Quadric fibrations} Our first main result is the following:
\begin{theorem}\label{thm:main}
Let $q\colon Q \to B$ a flat quadric fibration of relative dimension $d-2$. Assume that $B$ and $Q$ are $k$-smooth, that all the fibers of $q$ have corank $\leq 1$, and that the locus $D \subset B$ of the critical values of the fibration $q$ is $k$-smooth. Under these assumptions, the following holds:
\begin{itemize}
\item[(i)] When $d$ is even, we have $\mathrm{S}(Q) \Leftrightarrow \mathrm{S}(B) + \mathrm{S}(\widetilde{B})$, where $\widetilde{B}$ stands for the discriminant $2$-fold cover of $B$ (ramified over $D$).
\item[(ii)] When $d$ is odd and $\mathrm{char}(k)\neq 2$, we have $\{\mathrm{S}(V_i)\} + \{\mathrm{S}(\widetilde{D}_i)\} \Rightarrow \mathrm{S}(Q)$, where $V_i$ is any affine open of $B$ and $\widetilde{D}_i$ is any Galois $2$-fold cover of $D_i:=D\cap V_i$.
\end{itemize}
\end{theorem}
To the best of the author's knowledge, Theorem \ref{thm:main} is new in the literature. Intuitively speaking, it relates the Schur-finiteness conjecture for the total space $Q$ with the Schur-finiteness conjecture for certain coverings/subschemes of the base $B$. Among other ingredients, its proof makes use of Kontsevich's noncommutative mixed motives of twisted root stacks; consult \S\ref{sec:root}-\ref{sec:proof} below for details.

Making use of Theorem \ref{thm:main}, we are now able to prove the Schur-finiteness conjecture in new cases. Here are two low-dimensional examples:
\begin{corollary}[Quadric fibrations over curves]\label{cor:main}
Let $q\colon Q \to B$ be a quadric fibration as in Theorem \ref{thm:main} with $B$ a curve\footnote{Since $B$ is a curve, the locus $D\subset B$ of the critical values of $q$ is necessarily $k$-smooth.}. In this case, the conjecture $\mathrm{S}(Q)$ holds.
\end{corollary}
\begin{corollary}[Quadric fibrations over surfaces]\label{cor:main2}
Let $q\colon Q \to B$ be a quadric fibration as in Theorem \ref{thm:main} with $B$ a surface and $d$ odd. In this case, $\mathrm{S}(B)\Rightarrow\mathrm{S}(Q)$.
\end{corollary}
\begin{proof}
Given a smooth $k$-surface $X$, we have $\mathrm{S}(X)\Leftrightarrow \mathrm{S}(U)$ for any open $U$ of $X$. Therefore, thanks to Theorem \ref{thm:main}(ii), the proof follows from the fact that when $B$ is a surface, the conjectures $\{\mathrm{S}(V_i)\}$ can be replaced by the~conjecture~$\mathrm{S}(B)$.
\end{proof}
Corollary \ref{cor:main2} can be applied, for example, to the case where $B$ is (an open subscheme of) an abelian surface or a smooth projective surface with $p_g=0$ which satisfies Bloch's conjecture (see Guletskii-Pedrini \cite[\S4 Thm.~7]{GP}). Recall that Bloch's conjecture holds for surfaces not of general type (see Bloch-Kas-Leiberman \cite{BKL}), for surfaces which are rationally dominated by a product of curves (see Kimura \cite{Kimura}), for Godeaux, Catanese and Barlow surfaces (see Voisin \cite{Voisin2, Voisin}),~etc.
\begin{remark}[Related work]
Let $q\colon Q \to B$ be a quadric fibration as in Theorem \ref{thm:main}. In the particular case where $Q$ and $B$ are smooth {\em projective}, Bouali \cite{Bouali} and Vial \cite[\S4]{Vial} ``computed'' the Chow motive $\mathfrak{h}(Q)_\bbQ$ of $Q$ using smooth projective $k$-schemes of dimension $\leq \mathrm{dim}(B)$. Since the category of Chow motives (with $\bbQ$-coefficients) embeds fully-faithfully into $\mathrm{DM}_{\mathrm{gm}}(k)_\bbQ$ (see \cite[\S4]{Voevodsky}), these computations lead to an alternative ``geometric'' proof of Corollaries \ref{cor:main}-\ref{cor:main2}. Note that in Theorem \ref{thm:main} and in Corollaries \ref{cor:main}-\ref{cor:main2} we do {\em not} assume that $Q$ and $B$ are projective; we are (mainly) interested in geometric mixed motives and {\em not} in pure motives. 
\end{remark}
\subsection*{Intersections of quadrics}
Let $Y\subset \bbP^{d-1}$ be a smooth complete intersection of $m$ quadric hypersurfaces. The linear span of these quadrics gives rise to a flat quadric fibration $q\colon Q \to \bbP^{m-1}$ of relative dimension $d-2$, with $Q$ $k$-smooth. Under these notations, our second main result is the following:
\begin{theorem}\label{thm:intersection}
We have $\mathrm{S}(Q) \Rightarrow \mathrm{S}(Y)$. When $2m\leq d$, the converse also holds.
\end{theorem}
By combining Theorem \ref{thm:intersection} with the above Corollaries \ref{cor:main}-\ref{cor:main2}, we hence obtain a proof of the Schur-finiteness conjecture in the following cases:
\begin{corollary}[Intersections of two or three quadrics]\label{cor:intersection}
Assume that $q\colon Q \to \bbP^{m-1}$ is as in Theorem \ref{thm:main}. In this case, the conjecture $\mathrm{S}(Y)$ holds when $Y$ is a smooth complete intersection of two, or of three odd-dimensional, quadric hypersurfaces.\end{corollary}
\subsection*{Families of sextic du Val del Pezzo surfaces}
Recall that a {\em sextic du Val del Pezzo surface $X$} is a projective $k$-scheme with at worst du Val singularities and ample anticanonical class such that $K_X^2=6$. Consider a {\em family of sextic du Val del Pezzo surfaces $f\colon T \to B$}, \ie a flat morphism $f$ such that for every geometric point $x \in B$ the associated fiber $T_x$ is a sextic du Val del Pezzo surface. Following Kuznetsov \cite[\S5]{Pezzo}, given $d \in \{2,3\}$, let us write $\cM_d$ for the relative moduli stack of semistable sheaves on fibers of $T$ over $B$ with Hilbert polynomial $h_d(t):=(3t+d)(t+1)$, and $Z_d$ for the coarse moduli space of $\cM_d$. By construction, we have finite flat morphisms $Z_2 \to B$ and $Z_3 \to B$ of degrees $3$ and $2$, respectively. Under these notations, our third main result is the following:
\begin{theorem}\label{thm:Pezzo}
Let $f\colon T \to B$ be a family of sextic du Val del Pezzo surfaces. Assume that $\mathrm{char}(k)\not\in \{2,3\}$ and that $T$ is $k$-smooth. Under these assumptions, we have the equivalence of conjectures $\mathrm{S}(T) \Leftrightarrow \mathrm{S}(B) + \mathrm{S}(Z_2) + \mathrm{S}(Z_3)$.
\end{theorem} 
To the best of the author's knowledge, Theorem \ref{thm:Pezzo} is new in the literature. It leads to a proof of the Schur-finiteness conjecture in new cases. Here is an example:
\begin{corollary}[Families of sextic du Val del Pezzo surfaces over curves]\label{cor:Pezzo}
Let $f\colon T \to B$ be a family of sextic du Val del Pezzo surfaces as in Theorem \ref{thm:Pezzo} with $B$ a curve. In this case, the conjecture $\mathrm{S}(T)$ holds.
\end{corollary}
\begin{remark}
Let $f\colon T \to B$ be a family of sextic du Val del Pezzo surfaces as in Theorem \ref{thm:Pezzo}. To the best of the author's knowledge, the associated geometric mixed motive $M(T)_\bbQ$ has {\em not} been ``computed'' (in any non-trivial particular case). Nevertheless, consult Helmsauer \cite{thesis} for the ``computation'' of the Chow motive $\mathfrak{h}(X)_\bbQ$ of certain {\em smooth} (projective) del Pezzo surfaces $X$.  
\end{remark}
%----------------------------------------------------
\subsection*{Bass-finiteness conjecture}
%----------------------------------------------------
Let $k$ be a finite base field and $X$ a smooth $k$-scheme of finite type. The Bass-finiteness conjecture $\mathrm{B}(X)$ (see \cite[\S9]{Bass}) is one of the oldest and most important conjectures in algebraic $K$-theory. It asserts that the algebraic $K$-theory groups $K_n(X), n \geq 0$, are finitely generated. In the same vein, given an integer $r \geq 2$, we can consider the conjecture $\mathrm{B}(X)_{1/r}$, where $K_n(X)$ is replaced by $K_n(X)_{1/r}:=K_n(X) \otimes \bbZ[1/r]$. Our fourth main result is the following:
\begin{theorem}\label{thm:Bass}
The following holds:
\begin{itemize}
\item[(i)] Theorem \ref{thm:main} and Corollaries \ref{cor:main}-\ref{cor:main2} hold\footnote{Corollary \ref{cor:main2} (for the conjecture $\mathrm{B}(-)_{1/2}$) can also be applied to the case where $B$ is (an open subscheme of) an abelian surface; see \cite[Cor.~70 and Thm.~82]{Kahn}.} similarly for the conjecture $\mathrm{B}(-)_{1/2}$. In Corollary \ref{cor:main}, the groups $K_n(Q)_{1/2}, n \geq 2$, are moreover finite.
\item[(ii)] Theorem \ref{thm:intersection} holds similarly for the conjecture $\mathrm{B}(-)$. 
\item[(iii)] Corollary \ref{cor:intersection} holds similarly for the conjecture $\mathrm{B}(-)_{1/2}$. In the case where $Y$ is a smooth complete intersection of two quadric hypersurfaces, the groups $K_n(Y)_{1/2}, n \geq 2$, are moreover finite.
\item[(iv)] Theorem \ref{thm:Pezzo} and Corollary \ref{cor:Pezzo} hold similarly for the conjecture $\mathrm{B}(-)_{1/6}$. In Corollary \ref{cor:Pezzo}, the groups $K_n(T)_{1/6}, n \geq 2$, are moreover finite.
\end{itemize}
\end{theorem}
%-------------------------------------------------------------------------------
\section{Preliminaries}
%-------------------------------------------------------------------------------
In what follows, all schemes/stacks are of finite type over the perfect base~field~$k$.
\subsection*{Dg categories} For a survey on dg categories we invite the reader to consult \cite{ICM-Keller}. In what follows, we will write $\dgcat(k)$ for the category of (essentially small) dg categories and dg functors. Every (dg) $k$-algebra $A$ gives naturally rise to a dg category with a single object. Another source of examples is provided by schemes/stacks. Given a $k$-scheme $X$ (or stack $\cX$), the category of perfect complexes of $\cO_X$-modules $\perf(X)$ admits a canonical dg enhancement $\perf_\dg(X)$; consult \cite[\S4.6]{ICM-Keller}\cite{LO} for details. More generally, given a sheaf of $\cO_X$-algebras $\cF$, we can consider the dg category of perfect complexes of $\cF$-modules $\perf_\dg(X;\cF)$.

\subsection*{Noncommutative mixed motives} For a book, resp. recent survey, on noncommutative motives we invite the reader to consult \cite{book}, resp. \cite{survey}. Recall from~\cite[\S8.5.1]{book} (see also \cite{Miami, finMot, IAS}) the definition of Kontsevich's triangulated category of noncommutative mixed motives $\mathrm{NMot}(k)$. By construction, this category is idempotent complete, symmetric monoidal, and comes equipped with a $\otimes$-functor $U \colon \dgcat(k) \to \mathrm{NMot}(k)$. In what follows, given a $k$-scheme $X$ (or stack $\cX$) equipped with a sheaf of $\cO_X$-algebras $\cF$, we will write $U(X;\cF):=U(\perf_\dg(X;\cF))$.

\section{Noncommutative mixed motives of twisted root stacks}\label{sec:root}
Let $X$ be a $k$-scheme, $\cL$ a line bundle on $X$, $\sigma \in \Gamma(X,\cL)$ a global section, and $r>0$ an integer. In what follows, we will write $D\subset X$ for the zero locus of $\sigma$. Recall from \cite[Def.~2.2.1]{Codman} (see also \cite[Appendix B]{GW}) that the associated {\em root stack} $\cX$ is defined as the following fiber-product of algebraic stacks
$$
\xymatrix{
\cX:=\sqrt[r]{(\cL,\sigma)/X} \ar[d]_-p \ar[r] & [\bbA^1/\bbG_m] \ar[d]^-{\theta_r} \\
X \ar[r]_-{(\cL,\sigma)} & [\bbA^1/\bbG_m]\,,
}
$$
where $\theta_r$ stands for the morphism induced by the $r^{\mathrm{th}}$ power maps on $\bbA^1$ and $\bbG_m$. A {\em twisted root stack} $(\cX;\cF)$ consists of a root stack $\cX$ equipped with a sheaf of Azumaya algebras $\cF$. In what follows, we will write $s$ for the product of the ranks of $\cF$ (at each one of the connected components of $\cX$). The following result, which is of independent interest, will play a key role in the proof of Theorem \ref{thm:main}.
\begin{theorem}\label{thm:aux}
Assume that $X$ and $D$ are $k$-smooth.
\begin{itemize}
\item[(i)] We have an isomorphism $U(\cX)\simeq U(X) \oplus U(D)^{\oplus (r-1)}$.
\item[(ii)] Assume moreover that $\mathrm{char}(k)\neq r$ and that $k$ contains the $r^{\mathrm{th}}$ roots of unity. Under these extra assumptions, $U(\cX;\cF)_{1/rs}$ belongs to the smallest thick triangulated subcategory of $\NMot(k)_{1/rs}$ containing the noncommutative mixed motives $\{U(V_i)_{1/rs}\}$ and $\{U(\widetilde{D}_i^l)_{1/rs}\}$, where $V_i$ is any affine open subscheme of $X$ and $\widetilde{D}_i^l$ is any Galois $l$-fold cover of $D_i:=D\cap V_i$ with $l\nmid r$ and $l\neq 1$.
\end{itemize}
\end{theorem}
\begin{proof}
We start by proving item (i). Following \cite[Thm.~1.6]{Ueda}, the pull-back functor $p^\ast$ is fully-faithful and we have the following semi-orthogonal decomposition\footnote{Consult \cite{BO, BO1} for the definition of semi-orthogonal decomposition.} $\perf(X)=\langle \perf(D)_{r-1},\ldots,\perf(D)_1, p^\ast(\perf(X)\rangle$. All the categories $\perf(D)_j$ are equivalent (via a Fourier-Mukai type functor) to $\perf(D)$. Therefore, since the functor $U\colon \dgcat(k) \to \NMot(k)$ sends semi-orthogonal decompositions to direct sums, we obtain the searched direct sum decomposition $U(\cX) \simeq U(X) \oplus U(D)^{\oplus (r-1)}$.

Let us now prove item (ii). We consider first the particular case where $X=\mathrm{Spec}(A)$ is affine and the line bundle $\cL =\cO_X$ is trivial. Let $\mu_r$ be the group of $r^{\mathrm{th}}$ roots of unity and $\chi \colon \mu_r \to k^\times$ a (fixed) primitive character. Under these notations, consider the global quotient $[\mathrm{Spec}(B)/\mu_r]$, where $B:=A[t]/\langle t^r - \sigma\rangle$ and the $\mu_r$-action on $B$ is given by $g\cdot t:= \chi(g)^{-1} t$ for every $g \in \mu_r$ and by $g \cdot a := a$ for every $a \in A$. As explained in \cite[Example~2.4.1]{Codman}, the root stack $\cX$ agrees, in this particular case, with the global quotient $[\mathrm{Spec}(B)/\mu_r]$. By construction, the induced map $\mathrm{Spec}(B) \to X$ is a $r$-fold cover ramified over $D \subset X$. Moreover, for every $l$ such that $l\mid r$ and $l\neq 1$, the associated closed subscheme $\mathrm{Spec}(B)^{\mu_l}$ agrees with the ramification divisor $D \subset \mathrm{Spec}(B)$. Therefore, since the functor $U(-)_{1/rs}\colon \dgcat(k) \to \NMot(k)_{1/rs}$ is an additive invariant of dg categories in the sense of \cite[Def.~2.1]{book} (see \cite[\S8.4.5]{book}), we conclude from \cite[Cor.~1.28(ii)]{Orbifold} that, in this particular case, $U(\cX; \cF)_{1/rs}$ belongs to the smallest thick additive subcategory of $\NMot(k)_{1/rs}$ containing the noncommutative mixed motives $U(\mathrm{Spec}(B))^{\mu_l}_{1/rs}$ and $\{U(\widetilde{D}^l)_{1/rs}\}$, where $\widetilde{D}^l$ is any Galois $l$-fold cover of $D$ with $l \nmid r$ and $l\neq 1$. Furthermore, since the geometric quotient $\mathrm{Spec}(B)/\!\!/\mu_r$ agrees with $X$ and the latter scheme is $k$-smooth, \cite[Thm.~1.22]{Orbifold} implies that $U(\mathrm{Spec}(B))^{\mu_l}_{1/rs}$ is isomorphic to $U(X)_{1/rs}$. This finishes the proof of item (ii) in the particular case where $X$ is affine and the line bundle $\cL$ is trivial.

Let us now prove item (ii) in the general case. As explained above, given any affine open subscheme $V_i$ of $X$ which trivializes the line bundle $\cL$, the noncommutative mixed motive $U(\cV_i;\cF_i)_{1/rs}$, with $\cV_i:=p^{-1}(V_i)$ and $\cF_i:=\cF_{|\cV_i}$, belongs to the smallest thick additive subcategory of $\NMot(k)_{1/rs}$ containing $U(V_i)_{1/rs}$ and $\{U(\widetilde{D}^l_i)_{1/rs}\}$, where $\widetilde{D}^l_i$ is any Galois $l$-fold cover of $D_i:=D \cap V_i$ with $l\mid r$ and $l\neq 1$. Let us then choose an affine open cover $\{W_i\}$ of $X$ which trivializes the line bundle $\cL$. Since $X$ is quasi-compact (recall that $X$ is of finite type over $k$), this affine open cover admits a {\em finite} subcover. Consequently, the proof follows by induction from the $\bbZ[1/rs]$-linearization of the distinguished triangles of Lemma \ref{lem:key} below.
\end{proof}
\begin{lemma}\label{lem:key}
Given an open cover $\{W_1, W_2\}$ of $X$, we have an induced Mayer-Vietoris distinguished triangle of noncommutative mixed motives
\begin{equation}\label{eq:triangle}
U(\cX;\cF) \too U(\cW_1;\cF_1) \oplus U(\cW_2;\cF_2) \stackrel{\pm}{\too} U(\cW_{12}; \cF_{12}) \stackrel{\partial}{\too} \Sigma U(\cX;\cF)\,,
\end{equation}
where $\cW_{12}:=\cW_1 \cap \cW_2$ and $\cF_{12}:=\cF_{|\cW_{12}}$. 
\end{lemma}
\begin{proof}
Consider the following commutative diagram of dg categories
$$
\xymatrix{
\perf_\dg(\cX;\cF)_\cZ \ar[d] \ar[r] & \perf_\dg(\cX;\cF) \ar[d] \ar[r] & \perf_\dg(\cW_1;\cF_1) \ar[d] \\
\perf_\dg(\cW_2;\cF_2)_\cZ \ar[r] & \perf_\dg(\cW_2;\cF_2) \ar[r] & \perf_\dg(\cW_{12}; \cF_{12})\,,
}
$$
where $\cZ$ stands for the closed complement $\cX- \cW_1 = \cW_2 - \cW_{12}$ and $\perf_\dg(\cX;\cF)_\cZ$, resp. $\perf_\dg(\cW_2; \cF_2)_\cZ$, stands for the full dg subcategory of $\perf_\dg(\cX;\cF)$, resp. $\perf_\dg(\cW_2;\cF_2)$, consisting of those perfect complexes of $\cF$-modules, resp. $\cF_2$-modules, that are supported on $\cZ$. Both rows are short exact sequences of dg categories in the sense of Drinfeld/Keller (see \cite[\S4.6]{ICM-Keller}) and the left vertical dg functor is a Morita equivalence. Therefore, since the functor $U\colon \dgcat(k) \to \NMot(k)$ is a localizing invariant of dg categories in the sense of \cite[\S8.1]{book}, we obtain the following induced morphism of distinguished triangles:
$$
\xymatrix@C=1.7em@R=2.5em{
U(\perf_\dg(\cX;\cF)_\cZ) \ar[d]^-\simeq \ar[r] & U(\cX;\cF) \ar[d] \ar[r] & U(\cW_1;\cF_1) \ar[d] \ar[r]^-\partial & \Sigma U(\perf_\dg(\cX;\cF)_\cZ) \ar[d]^-\simeq \\
U(\perf_\dg(\cW_2;\cF_2)_\cZ) \ar[r] & U(\cW_2;\cF_2) \ar[r] & U(\cW_{12};\cF_{12}) \ar[r]^-\partial & \Sigma U(\perf_\dg(\cW_2;\cF_2)_\cZ)\,.
}
$$
Finally, since the middle square is homotopy (co)cartesian, we hence obtain the claimed Mayer-Vietoris distinguished triangle \eqref{eq:triangle}.
\end{proof}
%-------------------------------------------------------------------------------
\section{Proof of Theorem \ref{thm:main}}\label{sec:proof}
%-------------------------------------------------------------------------------
Following \cite[\S3]{Quadrics} (see also \cite[\S1.2]{ABB}), let $E$ be a vector bundle of rank $d$ on $B$, $q'\colon \bbP(E) \to B$ the projectivization of $E$ on $B$, $\cO_{\bbP(E)}(1)$ the Grothendieck line bundle on $\bbP(E)$, $\cL$ a line bundle on $B$, and finally $\rho\in \Gamma(B, S^2(E^\vee) \otimes \cL^\vee) = \Gamma(\bbP(E), \cO_{\bbP(E)}(2)\otimes \cL^\vee)$ a global section. Given this data, recall that $Q\subset \bbP(E)$ is defined as the zero locus of $\rho$ on $\bbP(E)$ and that $q\colon Q \to B$ is the restriction of $q'$ to $Q$; note that the relative dimension of $q$ is equal to $d-2$. Consider also the discriminant global section $\mathrm{disc}(q) \in \Gamma(B, \mathrm{det}(E^\vee)^{\otimes 2} \otimes (\cL^\vee)^{\otimes d})$ and the associated zero locus $D\subset B$; note that $D$ agrees with the locus of the critical values of $q$. 

Recall from \cite[\S3.5]{Quadrics} (see also \cite[\S1.6]{ABB}) that when $d$ is even, we can consider the {\em discriminant cover} $\widetilde{B}:=\mathrm{Spec}_B(Z(\cC l_0(q)))$ of $B$, where $Z(\cC l_0(q))$ stands for the center of the sheaf $\cC l_0(q)$ of even parts of the Clifford algebra associated to $q$; see \cite[\S3]{Quadrics} (and also \cite[\S1.5]{ABB}). By construction, $\widetilde{B}$ is a $2$-fold cover ramified over $D$. Moreover, since $D$ is $k$-smooth, $\widetilde{B}$ is also $k$-smooth.

Recall from \cite[\S3.6]{Quadrics} (see also \cite[\S1.7]{ABB}) that when $d$ is odd and $\mathrm{char}(k)\neq 2$, we can consider the {\em discriminant stack} $\cX:=\sqrt[2]{(\mathrm{det}(E^\vee)^{\otimes 2} \otimes (\cL^\vee)^{\otimes d}, \mathrm{disc}(q))/B}$. Since $\mathrm{char}(k)\neq 2$, $\cX$ is a Deligne-Mumford stack with coarse moduli space $B$.
\begin{proposition}\label{prop:computation}
Under the above notations, and assumptions, the following holds:
\begin{itemize}
\item[(i)] When $d$ is even, we have $U(Q)_{1/2} \simeq U(\widetilde{B})_{1/2} \oplus U(B)_{1/2}^{\oplus (d-2)}$.
\item[(ii)] When $d$ is odd and $\mathrm{char}(k)\neq 2$, $U(Q)_{1/2}$ belongs to the smallest thick triangulated subcategory of $\NMot(k)_{1/2}$ containing the noncommutative mixed motives $\{U(V_i)_{1/2}\}$ and $\{U(\widetilde{D}_i)_{1/2}\}$, where $V_i$ is any affine open subscheme of $B$ and $\widetilde{D}_i$ is any Galois $2$-fold cover of $D_i:=D\cap V_i$.
\end{itemize}
\end{proposition}
\begin{proof}
As proved in \cite[Thm.~4.2]{Quadrics} (see also \cite[Thm.~2.2.1]{ABB}), we have the semi-orthogonal decomposition $\perf(Q) = \langle \perf(B; \cC l_0(q)), \perf(B)_1, \ldots, \perf(B)_{d-2}\rangle$, where $\perf(B)_j:=q^\ast(\perf(B)) \otimes \cO_{Q/B}(j)$. All the categories $\perf(B)_j$ are equivalent (via a Fourier-Mukai type functor) to $\perf(B)$. Therefore, since the functor $U\colon \dgcat(k) \to \NMot(k)$ sends semi-orthogonal decompositions to direct sums, we obtain the direct sum decomposition $U(Q) \simeq U(B;\cC l_0(q)))\oplus U(B)^{\oplus (d-2)}$. 

We start by proving item (i). As explained in \cite[\S3.5]{Quadrics} (see also \cite[\S1.6]{ABB}), when $d$ is even, the category $\perf(B; \cC l_0(q))$ is equivalent (via a Fourier-Mukai type functor) to $\perf(\widetilde{B}; \cF)$ where $\cF$ is a certain sheaf of Azumaya algebras on $\widetilde{B}$ of rank $2^{\frac{d}{2}-1}$. This leads to an isomorphism $U(B; \cC l_0(q))\simeq U(\widetilde{B}; \cF)$. Making use of \cite[Thm.~2.1]{Azumaya}, we hence conclude that $U(B; \cC l_0(q))_{1/2}$ is isomorphic to $U(\widetilde{B}; \cF)_{1/2}\simeq U(\widetilde{B})_{1/2}$. Consequently, we obtain the isomorphism of item (i).

Let us now prove item (ii). As explained in \cite[\S3.6]{Quadrics} (see also \cite[\S1.7]{ABB}), when $d$ is odd, the category $\perf(B; \cC l_0(q))$ is equivalent (via a Fourier-Mukai type functor) to $\perf(\cX; \cF)$ where $\cF$ is a certain sheaf of Azumaya algebras on $\cX$ of rank $2^{\frac{d-1}{2}}$. This leads to an isomorphism $U(B; \cC l_0(q))\simeq U(\cX; \cF)$. By combining Theorem \ref{thm:aux}(ii) with the isomorphism $U(Q) \simeq U(\cX;\cF) \oplus U(B)^{\oplus (d-2)}$, we hence conclude that $U(Q)_{1/2}$ belongs to the smallest thick triangulated subcategory of $\NMot(k)_{1/2}$ containing $U(B)_{1/2}$, $\{ U(V_i)_{1/2}\}$, and $\{U(\widetilde{D}_i)_{1/2}\}$, where $V_i$ is any affine open subscheme of $B$ and $\widetilde{D}_i$ is any Galois $2$-fold cover of $D_i$. We now claim that $U(B)_{1/2}$ belongs to the smallest thick triangulated subcategory of $\NMot(k)_{1/2}$ containing $\{U(V_i)_{1/2}\}$; note that this would conclude the proof. Choose an affine open cover $\{W_i\}$ of $B$. Since $B$ is quasi-compact (recall that $B$ is of finite type over $k$), this affine open cover admits a {\em finite} subcover. Therefore, similarly to the proof of Theorem \ref{thm:aux}, our claim follows from an inductive argument using the $\bbZ[1/2]$-linearization of the Mayer-Vietoris distinguished triangles $ U(B) \to U(W_1) \oplus U(W_2) \stackrel{\pm}{\to} U(W_{12}) \stackrel{\partial}{\to} \Sigma U(B)$. 
\end{proof}
As proved in \cite[Thm.~2.8]{Bridge}, there exists a $\bbQ$-linear, fully-faithful, $\otimes$-functor $\Phi$ making the following diagram commute
\begin{equation}\label{eq:diagram-big}
\xymatrix{
\mathrm{Sm}(k) \ar[rrr]^-{X\mapsto \perf_\dg(X)} \ar[d]_-{M(-)_\bbQ} &&& \dgcat(k) \ar[d]^-{U(-)_\bbQ} \\
\mathrm{DM}_{\mathrm{gm}}(k)_\bbQ \ar[d]_-\pi &&& \mathrm{NMot}(k)_\bbQ \ar[d]^-{\underline{\mathrm{Hom}}(-,U(k)_\bbQ)}\\
\mathrm{DM}_{\mathrm{gm}}(k)_\bbQ/_{\!-\otimes \bbQ(1)[2]} \ar[rrr]_-{\Phi} &&& \mathrm{NMot}(k)_\bbQ\,,
}
\end{equation}
where $\mathrm{DM}_{\mathrm{gm}}(k)_\bbQ/_{\!-\otimes \bbQ(1)[2]}$ stands for the orbit category with respect to the Tate motive $\bbQ(1)[2]$ and $\uHom(-,-)$ for the internal Hom of the monoidal structure; note that the functors $X \mapsto \perf_\dg(X)$ and $\uHom(-,U(k)_\bbQ)$ are contravariant. By construction, $\pi$ is a faithful $\otimes$-functor. Therefore, it follows from \cite[Lem.~1.11]{Mazza} that we have the following equivalence:
\begin{equation}\label{eq:equivalence}
\mathrm{S}(X) \Leftrightarrow \text{the}\,\,\text{noncommutative}\,\,\text{mixed}\,\,\text{motive}\,\,(\Phi \circ \pi) (M(X)_\bbQ)\,\,\text{is}\,\,\text{Schur}\text{-}\text{finite}.
\end{equation}
We now have all the ingredients necessary to conclude the proof of Theorem \ref{thm:main}.
\subsection*{Item (i)}
The above functors $\pi$ and $\uHom(-,U(k)_\bbQ)$ are $\bbQ$-linear. Therefore, by combining Proposition \ref{prop:computation}(i) with the commutative diagram \eqref{eq:diagram-big}, we conclude that
\begin{equation}\label{eq:iso}
(\Phi\circ \pi)(M(Q)_\bbQ)\simeq (\Phi\circ \pi)(M(\widetilde{B})_\bbQ) \oplus (\Phi\circ \pi)(M(B)_\bbQ)^{\oplus (d-2)}\,.
\end{equation}
Since Schur-finiteness is stable under direct sums and direct summands, the proof of the equivalence $\mathrm{S}(Q) \Leftrightarrow \mathrm{S}(B) + \mathrm{S}(\widetilde{B})$ follows then from \eqref{eq:equivalence}-\eqref{eq:iso}.
\subsection*{Item (ii)} 
Recall from \cite[\S8.5.1-8.5.2]{book} that, by construction, $\NMot(k)_\bbQ$ is a $\bbQ$-linear closed symmetric monoidal triangulated category in the sense of Hovey \cite[\S6-7]{Hovey}. As proved in \cite[Thm.~1]{Guletskii}, this implies that Schur-finiteness has the 2-out-of-3 property with respect to distinguished triangles. The functor $\uHom(-,U(k)_\bbQ)$ is triangulated. Hence, by combining Proposition \ref{prop:computation}(ii) with the commutative diagram \eqref{eq:diagram-big}, we conclude that $(\Phi\circ \pi)(M(Q)_\bbQ)$ belongs to the smallest thick triangulated subcategory of $\NMot(k)_\bbQ$ containing the noncommutative mixed motives $\{(\Phi\circ \pi)(M(V_i)_\bbQ)\}$ and $\{(\Phi \circ \pi)(M(\widetilde{D}_i)_\bbQ)\}$, where $V_i$ is any affine open subscheme of $B$ and $\widetilde{D}_i$ is any Galois $2$-fold cover of $D_i$. Since by assumption the conjectures $\{\mathrm{S}(V_i)\}$ and $\{\mathrm{S}(\widetilde{D}_i)\}$ hold, \eqref{eq:equivalence} implies that the noncommutative mixed motives $\{(\Phi\circ \pi)(M(V_i)_\bbQ)\}$ and $\{(\Phi \circ \pi)(M(\widetilde{D}_i)_\bbQ)\}$ are Schur-finite. Therefore, making use of the 2-out-of-3 property of Schur-finiteness with respect to distinguished triangles (and of the stability of Schur-finiteness under direct summands), we conclude that $(\Phi\circ \pi)(M(Q)_\bbQ)$ is also Schur-finite. The proof follows now from the above equivalence \eqref{eq:equivalence}.
%-------------------------------------------------------------------------------
\section{Proof of Theorem \ref{thm:intersection}}
%-------------------------------------------------------------------------------
Recall from the proof of Proposition \ref{prop:computation} that we have the semi-orthogonal decomposition $\perf(Q)=\langle \perf(\bbP^{m-1};\cC l_0(q)), \perf(\bbP^{m-1})_1, \ldots, \perf(\bbP^{m-1})_{d-2} \rangle$, and consequently the following direct sum decompositon:
\begin{equation}\label{eq:direct}
U(Q)\simeq U(\bbP^{m-1}; \cC l_0(q)) \oplus U(\bbP^{m-1})^{\oplus (d-2)}\,.
\end{equation}
As proved in \cite[Thm.~5.5]{Quadrics} (see also \cite[Thm.~2.3.7]{ABB}), the following also holds:
\begin{itemize}
\item[(a)] When $2m<d$, we have $\perf(Y)=\langle \perf(\bbP^{m-1}; \cC l_0(q)), \cO(1), \ldots, \cO(d-2m)\rangle$. Consequently, since the functor $U\colon \dgcat(k) \to \NMot(k)$ sends semi-orthogonal decompositions to direct sums, we obtain the following direct sum decomposition $U(Y) \simeq U(\bbP^{m-1}; \cC l_0(q)) \oplus U(k)^{\oplus (d-2m)}$.
\item[(b)] When $2m=d$, the category $\perf(Y)$ is equivalence (via a Fourier-Mukai type functor) to $\perf(\bbP^{m-1}; \cC l_0(q))$. Consequently, we obtain an isomorphism of noncommutative mixed motives $U(Y) \simeq U(\bbP^{m-1}; \cC l_0(q))$.
\item[(c)] When $2m >d$, $\perf(Y)$ is an admissible subcategory of $\perf(\bbP^{m-1}; \cC l_0(q))$. This implies that $U(Y)$ is a direct summand of $U(\bbP^{m-1}; \cC l_0(q))$.
\end{itemize}
%If the conjecture $\mathrm{B}(Q)$ holds, then it follows from the above decomposition \eqref{eq:direct} that the algebraic $K$-theory groups $K_n(\perf_\dg(\bbP^{m-1}; \cC l_0(q))), n \geq 0$, are also finitely generated. Therefore, by combining the above descriptions (a)-(c) of the noncommutative mixed motive $U(Y)$ with the following isomorphisms (see \cite[\S8.6]{book})
%\begin{eqnarray*}
%\Hom_{\NMot(k)}(U(k), \Sigma^{-n} U(X))\simeq K_n(X) & n \in \bbZ & X \in \mathrm{Sm}(k)\,,
%\end{eqnarray*}
%we conclude that the conjecture $\mathrm{B}(Y)$ also holds. Note that when $2m\leq d$, a similar argument proves the converse implication $\mathrm{B}(Y) \Rightarrow \mathrm{B}(Q)$.
Let us now prove the implication $\mathrm{S}(Q) \Rightarrow \mathrm{S}(Y)$. If the conjecture $\mathrm{S}(Q)$ holds, then it follows from the decomposition \eqref{eq:direct}, from the commutative diagram \eqref{eq:diagram-big}, from the equivalence \eqref{eq:equivalence}, and from the stability of Schur-finiteness under direct summands, that the noncommutative mixed motive $\uHom(U(\bbP^{m-1};\cC l_0(q))_\bbQ, U(k)_\bbQ)$ is Schur-finite. Making use of the above descriptions (a)-(c) of $U(Y)$ and of the commutative diagram \eqref{eq:diagram-big}, we hence conclude that the noncommutative mixed motive $(\Phi \circ \pi)(M(Y)_\bbQ)$ is also Schur-finite. Consequently, the conjecture $\mathrm{S}(Y)$ follows now from the above equivalence \eqref{eq:equivalence}. Finally, note that when $2m\leq d$, a similar argument proves the converse implication $\mathrm{S}(Y) \Rightarrow \mathrm{S}(Q)$.
%-------------------------------------------------------------------------------
\section{Proof of Theorem \ref{thm:Pezzo}}
%-------------------------------------------------------------------------------
Recall first from \cite[Prop.~5.12]{Pezzo} that since $\mathrm{char}(k)\not\in \{2,3\}$ and $T$ is $k$-smooth, the $k$-schemes $B, Z_2$ and $Z_3$ are also $k$-smooth.
\begin{proposition}\label{prop:Pezzo}
We have $U(T)_{1/6} \simeq U(B)_{1/6} \oplus U(Z_2)_{1/6}\oplus U(Z_3)_{1/6}$.
\end{proposition}
\begin{proof}
As proved in \cite[Thm.~5.2 and Prop.~5.10]{Pezzo}, we have the semi-orthogonal decomposition $\perf(T) = \langle \perf(B), \perf(Z_2;\cF_2), \perf(Z_3; \cF_3)\rangle$, where $\cF_2$ (resp. $\cF_3$) is a certain sheaf of Azumaya algebras over $Z_2$ (resp. $Z_3$) of order $2$ (resp. $3$). Recall that the functor $U\colon \dgcat(k) \to \NMot(k)$ sends semi-orthogonal decompositions to direct sums. Therefore, we obtain the following direct sum decomposition:
\begin{equation}\label{eq:decomp-last}
U(T) \simeq U(B) \oplus U(Z_2; \cF_2) \oplus U(Z_3; \cF_3)\,.
\end{equation}
Since $\cF_2$ (resp. $\cF_3$) is of order $2$ (resp. $3$), the rank of $\cF_2$ (resp. $\cF_3$) is necessarily a power of $2$ (resp. $3$). Making use of \cite[Thm.~2.1]{Azumaya}, we hence conclude that the noncommutative mixed motive $U(Z_2;\cF_2)_{1/2}$ (resp. $U(Z_3; \cF_3)_{1/3}$) is isomorphic to $U(Z_2)_{1/2}$ (resp. $U(Z_3)_{1/3}$). Consequently, the proof follows now from the $\bbZ[1/6]$-linearization of \eqref{eq:decomp-last}.
\end{proof}
The functors $\pi$ and $\uHom(-,U(k)_\bbQ)$ in \eqref{eq:diagram-big} are $\bbQ$-linear. Therefore, similarly to the proof of item (i) of Theorem \ref{thm:main}, by combining Proposition \ref{prop:Pezzo} with the commutative diagram \eqref{eq:diagram-big}, we conclude that
\begin{equation}\label{eq:last}
(\Phi\circ \pi)(M(T)_\bbQ)\simeq (\Phi\circ \pi)(M(\widetilde{B})_\bbQ) \oplus (\Phi\circ \pi)(M(Z_2)_\bbQ)\oplus (\Phi\circ \pi)(M(Z_3)_\bbQ) \,.
\end{equation}
Since Schur-finiteness is stable under direct sums and direct summands, the proof follows then from the combination of \eqref{eq:last} with the above equivalence \eqref{eq:equivalence}.
%-------------------------------------------------------------------------------
\section{Proof of Theorem \ref{thm:Bass}}
%-------------------------------------------------------------------------------
\subsection*{Item (i)} We start by proving the first claim. As explained in \cite[\S8.6]{book} (see also \cite[Thm.~15.10]{Duke}), given $X \in \mathrm{Sm}(k)$, we have the isomorphisms of abelian groups:
\begin{eqnarray}\label{eq:iso-key}
\Hom_{\NMot(k)}(U(k), \Sigma^{-n} U(X))\simeq K_n(X) && n \in \bbZ\,.
\end{eqnarray}
Assume that $d$ is even. By combining Proposition \ref{prop:computation}(i) with the $\bbZ[1/2]$-linearization of  \eqref{eq:iso-key}, we conclude that $K_n(Q)_{1/2}\simeq K_n(\widetilde{B})_{1/2} \oplus K_n(B)_{1/2}^{\oplus (d-2)}$. Therefore, since finite generation is stable under direct sums and direct summands, we obtain the equivalence $\mathrm{B}(Q)_{1/2} \Leftrightarrow \mathrm{B}(B)_{1/2} + \mathrm{B}(\widetilde{B})_{1/2}$. Assume now that $d$ is odd and that $\mathrm{char}(k)\neq 2$. Finite generation has the 2-out-of-3 property with respect to (short or long) exact sequences and is stable under direct summands. Therefore, the proof of the implication $\{\mathrm{B}(V_i)_{1/2}\} + \{\mathrm{B}(\widetilde{D}_i)_{1/2}\} \Rightarrow \mathrm{B}(Q)_{1/2}$ follows from the combination of Proposition \ref{prop:computation}(ii) with the $\bbZ[1/2]$-linearization of \eqref{eq:iso-key}. Finally, recall from \cite{Quillen, Quillen2, Quillen1} that the conjecture $\mathrm{B}(X)$ holds in the case where $\mathrm{dim}(X)\leq 1$. Therefore, the Corollaries \ref{cor:main}-\ref{cor:main2} also hold similarly for the conjecture $\mathrm{B}(-)_{1/2}$. 

We now prove the second claim. Let $q\colon Q \to B$ be a quadric fibration as in Theorem \ref{thm:main} with $B$ a curve. Thanks to Corollary \ref{cor:main} (for the conjecture $\mathrm{B}(-)_{1/2}$), it suffices to show that the groups $K_n(Q), n \geq 2$, are torsion. Assume first that $d$ is even. By combining Proposition \ref{prop:computation}(i) with the $\bbQ$-linearization of \eqref{eq:iso-key}, we obtain an isomorphism $K_n(Q)_\bbQ \simeq K_n(\widetilde{B})_\bbQ \oplus K_n(B)_\bbQ^{\oplus (d-2)}$. Thanks to Proposition \ref{prop:curve} below, we have $K_n(\widetilde{B})_\bbQ=K_n(B)_\bbQ=0$ for every $n \geq 2$. Therefore, we conclude that the groups $K_n(Q), n \geq 2$, are torsion. Assume now that $d$ is even and that $\mathrm{char}(k)\neq 2$. Thanks to Proposition \ref{prop:computation}(ii), $U(Q)_\bbQ$ belongs to the smallest thick triangulated subcategory of $\NMot(k)_\bbQ$ containing the noncommutative mixed motives $\{U(V_i)_\bbQ\}$ and $\{U(\widetilde{D}_i)_\bbQ\}$, where $V_i$ is any affine open subscheme of $B$ and $\widetilde{D}_i$ is any Galois $2$-fold cover of $D_i$. Moreover, $U(Q)_\bbQ$ may be explicitly obtained from $\{U(V_i)_\bbQ\}$ and $\{U(\widetilde{D}_i)_\bbQ\}$ using solely the $\bbQ$-linearization of the Mayer-Vietoris distinguished triangles. Therefore, since $K_n(V_i)_\bbQ=0$ for every $n \geq 2$ (see Proposition \ref{prop:curve} below) and $K_n(\widetilde{D}_i)_\bbQ=0$ for every $n \geq 1$ (see Quillen's computation \cite{Quillen1} of the algebraic $K$-theory of a finite field), an inductive argument using the $\bbQ$-linearization of \eqref{eq:iso-key} and the  $\bbQ$-linearization of the Mayer-Vietoris distinguished triangles implies that the groups $K_n(Q), n \geq 2$, are torsion.
\begin{proposition}\label{prop:curve}
We have $K_n(X)_\bbQ=0$ for every $n \geq 2$ and smooth $k$-curve $X$.
\end{proposition}
\begin{proof}
In the particular case where $X$ is affine, this result was proved in \cite[Cor.~3.2.3]{Harder} (see also \cite[Thm.~0.5]{Quillen}). In the general case, choose an affine open cover $\{W_i\}$ of $X$. Since $X$ is quasi-compact, this affine open cover admits a {\em finite} subcover. Therefore, the proof follows from an inductive argument (similar to the one in the proof of Theorem \ref{thm:aux}(ii)) using the $\bbQ$-linearization of \eqref{eq:iso-key} and the $\bbQ$-linearization of the Mayer-Vietoris distinguished triangles.% $U(B) \to U(W_1) \oplus U(W_2) \stackrel{\pm}{\to} U(W_{12}) \stackrel{\partial}{\to} \Sigma U(B)$.
\end{proof}
\subsection*{Item (ii)} If the conjecture $\mathrm{B}(Q)$ holds, then it follows from the decomposition \eqref{eq:direct} and from the isomorphisms \eqref{eq:iso-key} that the algebraic $K$-theory groups $K_n(\perf_\dg(\bbP^{m-1}; \cC l_0(q))), n \geq 0$, are finitely generated. Therefore, by combining the descriptions (a)-(c) of the noncommutative mixed motive $U(Y)$ (see the proof of Theorem \ref{thm:intersection}) with \eqref{eq:iso-key}, we conclude that the conjecture $\mathrm{B}(Y)$ also holds. Note that when $2m\leq d$, a similar argument proves the converse implication $\mathrm{B}(Y) \Rightarrow \mathrm{B}(Q)$.

\subsection*{Item (iii)} Items (i)-(ii) of Theorem \ref{thm:Bass} imply that Corollary \ref{cor:intersection} holds similarly for the conjecture $\mathrm{B}(-)_{1/2}$. We now address the second claim. Let $q\colon Q \to \bbP^1$ be the quadric fibration associated to the smooth complete intersection $Y$ of two quadric hypersurfaces. Thanks to item (i), the groups $K_n(Q)_{1/2}, n \geq 2$, are finite. Therefore, making use of the decomposition \eqref{eq:direct}, of the $\bbZ[1/2]$-linearization of \eqref{eq:iso-key}, and of the above descriptions (a)-(c) of $U(Y)$ (see the proof of Theorem \ref{thm:intersection}), we conclude that the groups $K_n(Y)_{1/2}, n \geq 2$, are also finite.

\subsection*{Item (iv)}
We start by proving the first claim. By combining Proposition \ref{prop:Pezzo} with the $\bbZ[1/6]$-linearization of \eqref{eq:iso-key}, we conclude that 
$$K_n(T)_{1/6}\simeq K_n(B)_{1/6} \oplus K_n(Z_2)_{1/6} \oplus K_n(Z_3)_{1/6}\,.$$ 
Therefore, since finite generation is stable under sums and direct summands, we obtain the equivalence $\mathrm{B}(T)_{1/6} \Leftrightarrow \mathrm{B}(B)_{1/6} + \mathrm{B}(Z_2)_{1/6} + \mathrm{B}(Z_3)_{1/6}$. As mentioned in the proof of item (i), the conjecture $\mathrm{B}(X)$ holds in the case where $\mathrm{dim}(X)\leq 1$. Hence, Corollary \ref{cor:Pezzo} also holds similarly for the conjecture $\mathrm{B}(-)_{1/6}$.

We now prove the second claim. Let $f\colon T \to B$ be a family of sextic du Val del Pezzo surfaces as in Theorem \ref{thm:Pezzo} with $B$ a curve. Similarly to the proof of item (i) of Theorem \ref{thm:Bass}, it suffices to show that the groups $K_n(T), n \geq 2$, are torsion. By combining Proposition \ref{prop:Pezzo} with the $\bbQ$-linearization of \eqref{eq:iso-key}, we obtain an isomorphism $K_n(T)_\bbQ \simeq K_n(B)_\bbQ \oplus K_n(Z_2)_\bbQ \oplus K_n(Z_3)_\bbQ$. Thanks to Proposition \ref{prop:curve}, we have moreover $K_n(B)_\bbQ = K_n(Z_2)_\bbQ = K_n(Z_3)_\bbQ =0$ for every $n \geq 2$. Therefore, we conclude that the groups $K_n(T), n \geq 2$, are torsion.

\subsection*{Acknowledgments:} The author is grateful to Joseph Ayoub for useful e-mail exchanges concerning the Schur-finiteness conjecture. The author also would like to thank the Hausdorff Research Institute for Mathematics for its hospitality.

\end{document}

\end{proof}